\documentclass[11pt]{article}
\usepackage{amssymb}
\usepackage{amsfonts}
\usepackage{amsthm}
\usepackage{amsmath}
\usepackage{amscd}
\usepackage[T1]{fontenc}
\usepackage{t1enc}
\usepackage[mathscr]{eucal}
\usepackage{indentfirst}
\usepackage{graphicx}
\usepackage{graphics}
\usepackage{pict2e}
\usepackage{cite}
\usepackage{epic}
\usepackage{tikz}
\usepackage{mathtools}
\usetikzlibrary{decorations.pathreplacing}

\newtheorem{theo}{Theorem}[section]
\newtheorem*{theo*}{Theorem}
\newtheorem{definition}{Definition}[section]
\newtheorem{prop}[theo]{Proposition}
\newtheorem{lemma}[theo]{Lemma}
\newtheorem{coro}[theo]{Corollary}
\newtheorem{conj}[theo]{Conjecture}
\newtheorem{quest}[theo]{Question}

\newtheorem{ex}[theo]{Example}

\DeclarePairedDelimiter\ceil{\lceil}{\rceil}
\DeclarePairedDelimiter\floor{\lfloor}{\rfloor}
\begin{document}
\title{
On the $e$\nobreakdash-positivity of trees and spiders
}

\author{
Kai Zheng
\thanks
{Department of Mathematics, Massachusetts Institute of Technology,
Cambridge, MA 02142, USA.
Email: {\tt kzzheng@mit.edu}.}
}
\maketitle
\begin{abstract}
    We prove that the chromatic symmetric function of any tree with a vertex of degree at least six is not $e$\nobreakdash-positive, that is, it cannot be written as a nonnegative linear combination of elementary symmetric functions. This makes significant progress towards a recent conjecture of Dahlberg, She, and van Willigenburg, who conjectured the result for the chromatic symmetric functions of all trees with a vertex of degree at least four. We also provide a series of conditions that can identify when the chromatic symmetric function of a spider, a tree consisting of multiple paths all adjacent at a leaf to a center vertex, is not $e$\nobreakdash-positive. These conditions generalize to trees and graphs with cut vertices as well. Finally, by using a result of Orellana and Scott, we give a method to inductively calculate certain coefficients in the elementary symmetric function expansion of the chromatic symmetric function of a spider, leading to further $e$\nobreakdash-positivity conditions for spiders.
\end{abstract}

\section{Introduction}
In 1995, Richard Stanley introduced the chromatic symmetric function of a graph - a generalization of the well known chromatic polynomial \cite{S95}. Stanley's seminal paper included numerous properties of the chromatic symmetric function and an equivalent form of the Stanley-Stembridge conjecture that the chromatic symmetric function of the incomparability graph of any $(3 + 1)$-free poset is $e$\nobreakdash-positive, i.e.~a nonnegative linear combination of elementary symmetric functions \cite{SS}. This conjecture has led to numerous studies related to the classification of graphs with $e$\nobreakdash-positive chromatic symmetric function, or simply  $e$\nobreakdash-positive graphs for short \cite{DFvW, Lolly, DSvW, Foley, G1, G2, GS, GP, MMW, OS}. 

One difficulty of working with the chromatic symmetric function is that there is no deletion-contraction relation like that of the chromatic polynomial. Many have circumvented this problem, however, by using quasisymmetric functions, functions in noncommuting variables, and a vertex weighted version of the chromatic symmetric function \cite{CH, CS, GS, SW}. In each of these cases, the generalization of the chromatic symmetric function allows for some form of a deletion-contraction relation that can then be used to show properties that are inherited by the chromatic symmetric function. 

Another active area of research regards the question of whether the chromatic symmetric function can distinguish trees up to isomorphism \cite{LS, MMW, OS}. Using a probabilistic algorithm, the answer is known to be affirmative for trees of up to 29 vertices \cite{HJ}. The affirmative is also known for certain simplified families of trees such as spiders and caterpillars \cite{AZ, ADZ, LS, MMW}.

Recently, Dahlberg, She, and van Willigenburg have combined these two avenues and studied the $e$\nobreakdash-positivity of trees. This study is also in part motivated by Stanley's initial observation that being $e$\nobreakdash-positive is related to being claw-free and that trees are likely $e$\nobreakdash-positive only ``by accident''. In line with this observation, Dahlberg, She, and van Willigenburg show that many classes of trees are not $e$\nobreakdash-positive, and exhibit numerous conditions for the $e$\nobreakdash-positivity of trees and graphs with cut vertices. In particular, they give the following theorem and conjecture.

\begin{theo} [\!\! Dahlberg, She, and van Willigenburg \cite{DSvW}] \label{dahl bd}
Let $T$ be an $n$-vertex tree with a vertex of degree $d \geq 3$. If $d \geq \log_2{n} + 1,$ then $T$ is not $e$\nobreakdash-positive. 
\end{theo}

\begin{conj} [\!\! Dahlberg, She, and van Willigenburg \cite{DSvW}] \label{DSvW Conj}
Any tree with a vertex of degree at least four is not $e$\nobreakdash-positive.
\end{conj}
Dahlberg, She, and van Willigenburg obtain their results by reducing, in many cases, the $e$\nobreakdash-positivity of trees to that of spiders. Specifically, they show that certain classes of spiders are not $e$\nobreakdash-positive and then use a lemma to extend the same results to some corresponding classes of trees. In this paper, we follow the same strategy and give  further conditions for the $e$\nobreakdash-positivity of trees and spiders. Our main result is a partial resolution of Conjecture~\ref{DSvW Conj} to trees with a vertex of degree at least six, which we obtain via a similar reduction to spiders. Our paper is structured as follows. 

In Section \ref{sec: bg} we review necessary background on chromatic symmetric functions and notation that will be used throughout the paper. 

In Section \ref{sec: e-pos} we expand on many of the tests for $e$\nobreakdash-positivity from \cite{DSvW} and build up to our main result. We show in Theorem \ref{Th: mod-test} that the sum of the residues of the leg lengths of an $e$\nobreakdash-positive spider modulo any positive integer must be less than twice that integer. While this theorem is not used in our attempt at Conjecture~\ref{DSvW Conj}, it provides a simple test for $e$\nobreakdash-positivity and rules out the $e$\nobreakdash-positivity of many classes of spiders, which we discuss in Corollary \ref{Cor: var cond}. We then give a more complicated criterion for the $e$\nobreakdash-positivity of spiders that leads to our main result, stated in Theorem~\ref{new result}.

\begin{theo*}
Any spider with at least six legs is not $e$\nobreakdash-positive.
\end{theo*}

We then extend this theorem to trees and graphs with cut vertices by applying \cite[Lemmas~12~and~13]{DSvW}. This yields the partial resolution of Conjecture~\ref{DSvW Conj}.

\begin{theo*}
Any tree with a vertex of degree at least six is not $e$\nobreakdash-positive.
\end{theo*}
\begin{theo*}
Any graph with a cut vertex whose deletion produces a graph with at least six connected components is not $e$\nobreakdash-positive.
\end{theo*}

In Sections \ref{sec: small spiders} and \ref{sec: coeff calcs}, we turn to calculating specific coefficients in the chromatic symmetric functions of spiders and come up with further conditions for $e$\nobreakdash-positivity. 

In Section \ref{sec: small spiders} we focus on spiders with three or four legs. We present a decomposition for the chromatic symmetric function of spiders, in Lemma~\ref{lm: spi decomp}, in terms of the chromatic symmetric functions of spiders with fewer legs. By using this decomposition, we are able to calculate specific expansion coefficients of certain elementary symmetric functions. Very recently, Wang and Wang have also used this decomposition to further classify the spiders with three legs that are $e$\nobreakdash-positive \cite{Wang}.

In Section \ref{sec: coeff calcs} we apply this decomposition to spiders with more legs. Our main result of this section gives further conditions for when spiders with odd length legs are $e$\nobreakdash-positive. While in \cite{DSvW} it is found that spiders with at least $3$ odd length legs are not $e$\nobreakdash-positive, we find that spiders with at least four legs, exactly two odd length legs, and longest leg with even length, are also not $e$\nobreakdash-positive. 

Finally, in Section \ref{sec: fin} we conclude with questions and conjectures that suggest the $e$\nobreakdash-positivity of spiders is itself an interesting subject.

\section{Background} \label{sec: bg}
In this section we provide the necessary notation and background information. We start with our notation for partitions and spiders.

\begin{definition}
We say that $\lambda = (\lambda_1, \ldots, \lambda_{\ell(\lambda)})$ is a \emph{partition} of $n$, denoted by $\lambda \vdash n$, if $\lambda_1 + \cdots + \lambda_{\ell(\lambda)} = n$ and $\lambda_1 \geq \cdots \geq  \lambda_{\ell(\lambda)} > 0$.
\end{definition}

The individual $\lambda_i$'s are often called the \emph{parts} of the partition $\lambda$, and if there are $j$ parts equal to $k$, this is often abbreviated by $k^j$ in the partition. For example, $\lambda = (\lambda_1^{a_1}, \ldots , \lambda_{\ell(\lambda)}^{a_{\ell(\lambda)}})$ denotes the partition of $n = \sum_{i=1}^{\ell(\lambda)} a_i \lambda_i$ that has $a_i$ parts equal to $\lambda_i$ for each $i = 1, \ldots, \ell(\lambda)$.

Given an $n$-vertex graph $G = G(V, E)$, a \emph{connected partition} $C$ of $G$ is a partitioning of its vertices $V$ into sets $\{V_1, \ldots, V_k\}$ such that the induced subgraph of $G$ on each $V_i$ is connected. The \emph{type} of the connected partition $C$ is the partition of $n$ with parts equal to the sizes of the sets $V_i$ sorted in decreasing order. 

Our results in the next section focus on a specific family of graphs known as spiders. Two examples of spiders are shown in Example~\ref{ex: spiders}.

\begin{definition}
A spider with $d$ legs is a graph consisting of $d$ paths, each adjacent at a leaf to a center vertex. Given a partition $\lambda = (\lambda_1, \ldots, \lambda_d)$ of $n-1$, denoted by $\lambda \vdash n-1$, the spider $S(\lambda) = S(\lambda_1, \ldots, \lambda_d)$ is the $n$ vertex graph consisting of $d$ paths, $P_{\lambda_1}, \ldots, P_{\lambda_d}$, of orders $\lambda_1, \ldots, \lambda_d$ respectively, all connected to a vertex, $v$, of degree $d$. Each path $P_{\lambda_i}$ is called a leg of length $\lambda_i$ and consists of $\lambda_i$ vertices\footnote{We note that this differs from the usual definition of the length of a path, which refers the number of edges on the path. When referring to the length of a spider's leg, we mean the number of vertices on the leg. When discussing paths however, $P_n$ will be called the path of order $n$.}. The vertex $v$ is called the \emph{center}.
\end{definition}

\begin{ex} \label{ex: spiders}
The spiders $S(2,1,1)$ and $S(3,2,1)$. 
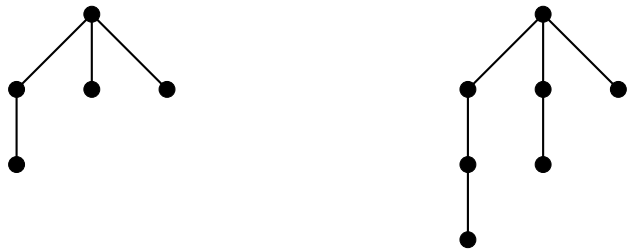
\vspace{1cm}
\begin{center}
\begin{tikzpicture}
%% vertices
\draw[fill=black] (-3,0) circle (3pt);
\draw[fill=black]  (-4,-1) circle (3pt);
\draw[fill=black]  (-4,-2) circle (3pt);
\draw[fill=black]  (-3,-1) circle (3pt);
\draw[fill=black]  (-2,-1) circle (3pt);

\draw[fill=black] (3,0) circle (3pt);
\draw[fill=black] (2,-1) circle (3pt);
\draw[fill=black] (2,-2) circle (3pt);
\draw[fill=black] (2,-3) circle (3pt);
\draw[fill=black] (3,-1) circle (3pt);
\draw[fill=black] (3,-2) circle (3pt);
\draw[fill=black] (4,-1) circle (3pt);
%%% edges
\draw[thick] (-3,0) -- (-4,-1) -- (-4,-2) (-3,0)--(-3,-1) (-3,0)--(-2,-1);

\draw[thick](3,0) -- (2,-1) -- (2,-2)--(2,-3) (3,0)--(3,-1)--(3,-2) (3,0)--(4,-1);
\end{tikzpicture}
\end{center}
\end{ex}

Next, we briefly review some background on the chromatic symmetric function and symmetric functions in general. For a more in depth treatment of symmetric functions we refer the reader to \cite{MacDonald} and \cite{Sagan}.

A function $f(x_1, x_2, \ldots) \in \mathbb{R}[x_1, x_2, \ldots]$ is said to be \emph{symmetric} if it is invariant under any permutation of its indeterminates. The algebra of symmetric functions is the graded algebra $\Lambda = \Lambda^0 \oplus \Lambda^1 \oplus \cdots$, where $\Lambda^d$ consists of the homogeneous degree $d$ symmetric functions, and is a subalgebra of $\mathbb{R}[x_1, x_2, \ldots]$. 

A common basis of $\Lambda$ are the \emph{elementary symmetric} functions. The $n$th elementary symmetric function, $e_n$, for $n \geq 1$ is given by
\begin{equation*}
e_n = \sum\limits_{j_1 < j_2 < \cdots < j_n} x_{j_1} \cdots x_{j_n},
\end{equation*}
while the elementary symmetric function, $e_{\lambda}$, for a partition $\lambda = (\lambda_1, \ldots,\lambda_{\ell(\lambda)})$, is 
\begin{equation*}
     e_\lambda = e_{\lambda_1} \cdots e_{\lambda_{\ell(\lambda)}}.
\end{equation*}
 A symmetric function is said to be \emph{$e$\nobreakdash-positive} if it can be written as a nonnegative linear combination of elementary symmetric functions. 

In this paper, the symmetric functions we use are chromatic symmetric functions of graphs. These are defined given a finite and simple graph as follows.
\begin{definition} [\!\! Stanley \cite{S95}]
Let $G$ be an $n$-vertex graph with vertex set $V_G = \{v_1, \ldots,$ $v_n\}$. The \emph{chromatic symmetric function} is defined by
\begin{equation*}
    X_G = \sum_{\kappa}x_{\kappa(v_1)} \cdots x_{\kappa(v_n)},
\end{equation*}
where the sum is over all proper colorings $\kappa: V_G \xrightarrow[]{} \{1, 2, \ldots \}$, i.e., colorings in which $\kappa(v) \neq \kappa(u)$ if $v$ and $u$ are adjacent.
\end{definition}

For simplicity, we say that a graph itself is $e$\nobreakdash-positive if its chromatic symmetric function is $e$\nobreakdash-positive. We use $[e_\lambda]X_G$ to denote the coefficient of $e_\lambda$ in the expansion of $X_G$ in the basis of elementary symmetric functions, and refer to these coefficients collectively as \emph{$e$\nobreakdash-coefficients}. Thus, a graph is $e$\nobreakdash-positive if and only if all of the $e$\nobreakdash-coefficients of its chromatic symmetric function are nonnegative.

\section{$e$\nobreakdash-positivity conditions for spiders} \label{sec: e-pos}
In this section, we give $e$\nobreakdash-positivity conditions for spiders and apply a lemma of \cite{DSvW}, which we state later, to generalize these results to trees. Throughout this section, when referring to a spider $S(\lambda_1, \ldots, \lambda_d)$, it is assumed that $\lambda_1 \geq \cdots \geq \lambda_d$. We rely on the following lemma to show that a spider is not $e$\nobreakdash-positive.

\begin{lemma} [\!\! Wolfgang \cite{W97}] \label{Wolf}  
If an $n$-vertex graph $G$ is $e$\nobreakdash-positive then it has a connected partition of type $\lambda$ for every $\lambda \vdash n$.
\end{lemma}

Thus, in order to show that a spider is not $e$-positive, it suffices show that it is missing a connected partition of some type. Dahlberg, She, and van Willigenburg use this strategy to present numerous $e$\nobreakdash-positivity conditions for spiders including the following two theorems in \cite{DSvW}. 

\begin{theo} [\!\! Dahlberg, She, and van Willigenburg \cite{DSvW}] \label{Th: Dahl-odd}
Every spider with at least three legs of odd length is not $e$\nobreakdash-positive.
\end{theo}
\begin{theo}[\!\! Dahlberg, She, and van Willigenburg \cite{DSvW}]  \label{Dahl-q}
Let $S = S(\lambda_1, \ldots, \lambda_d)$ be an $n$-vertex spider. For some $\lambda_i$ such that $2 \leq i < d$, let $n = q(\lambda_i + 1) + r$ where $0 \leq r < \lambda_i + 1$, $r = qd' + r'$, $0 \leq r' < q$, and $t = \lambda_{i+1} + \cdots + \lambda_d$. If $q \geq \frac{\lambda_i + 1}{t - 1}$ then $S$ is not $e$\nobreakdash-positive. In particular, it is missing a connected partition of type
\begin{equation*}
    (\lambda_i + d' + 2)^{r'}(\lambda_i + d' + 1)^{q-r'}.
\end{equation*}
\end{theo}
Using the same strategy, we start by generalizing the above theorems.
\begin{theo} \label{Th: mod-test}
Let $S = S(\lambda_1, \ldots, \lambda_d)$ be an $n$-vertex spider. For a positive integer $m > 1$, let $\lambda_i^{(m)}$ denote the residue of $\lambda_i$ modulo $m$. Let $n = mq + r$ with $0 \leq r < m$.
Then $S$ has a connected partition of type $(m^q, r)$ if and only if one of the following holds:
\begin{enumerate}
    \item $1+\lambda_1^{(m)} + \cdots + \lambda_d^{(m)} = r$ 
    \item $1+\lambda_1^{(m)} + \cdots + \lambda_d^{(m)} = m+r$ and $\lambda_i^{(m)} \geq r$ for some $i \in \{1, \dots, d \}$.
\end{enumerate}
In particular, if $1+\lambda_1^{(m)} + \cdots + \lambda_d^{(m)} \geq 2m$, or  $1+\lambda_1^{(m)} + \cdots + \lambda_d^{(m)} = m+r$ and $\lambda_i^{(m)} < r$ for each $i = 1, \ldots, d$, then $S$ is not $e$\nobreakdash-positive. 
\end{theo}
\begin{proof}
    For the forward direction, suppose $S$ has a connected partition $C$ of type $(m^q, r)$.
    
    If $r = 0$, then the vertices of each leg are partitioned into sets of size $m$ with $\lambda_i^{(m)}$ vertices left over for each leg of length $\lambda_i$. For each $i$, $\lambda_i^{(m)} < m$, so the $\lambda_i^{(m)}$ vertices must be in a set containing the center. Thus, there is a set in the connected partition $C$ consisting of $1 + \lambda_1^{(m)} + \cdots + \lambda_d^{(m)}$ vertices. It follows that $1 + \lambda_1^{(m)} + \cdots + \lambda_d^{(m)} = m = m + r$ and for some $i$, $\lambda_i^{(m)} \geq 0$.
    
    If $r > 0$, let $V_1 \in C$ be the set of vertices of size $r$. First suppose $V_1$ contains the center and let $t_i$ be the number of vertices on the leg of length $\lambda_i$ contained in $V_1$. Notice that $t_i \equiv \lambda_i^{(m)} \pmod{m}$ and, in particular, $t_i \geq \lambda_i^{(m)}$. It follows that $|V_1| = r = 1 + t_1 + \cdots + t_d \geq 1+\lambda_1^{(m)} + \cdots + \lambda_d^{(m)}$. However, as $1+\lambda_1^{(m)} + \cdots + \lambda_d^{(m)} \equiv r \pmod{m}$, it follows that $1+\lambda_1^{(m)} + \cdots + \lambda_d^{(m)} = r$, completing the forward direction.
    
    Now suppose that $V_1$, the set of vertices of size $r$, does not contain the center. Then, $V_1$ must only contain vertices from a single leg of the spider. Let this leg be of length $\lambda_i$.
    Since the rest of the vertices are all in sets of size $m$, it follows that the set of vertices containing the center is of size $1+\lambda_1^{(m)} + \cdots + (\lambda_i-r)^{(m)} + \cdots + \lambda_d^{(m)} = m$. We show that this implies one of the two stated conditions must hold. If $\lambda_i^{(m)} < r$, then $(\lambda_i-r)^{(m)} = \lambda_i^{(m)} + m-r$ and the previous equation simplifies to $1+\lambda_1^{(m)} + \cdots + \lambda_i^{(m)} + \cdots + \lambda_d^{(m)} = r$. If $\lambda_i^{(m)} \geq r$, then $(\lambda_i-r)^{(m)} = \lambda_i^{(m)} - r$, and $1+\lambda_1^{(m)} + \cdots + \lambda_i^{(m)} + \cdots + \lambda_d^{(m)} = m+r$. 
    
    For the reverse direction, if $1+\lambda_1^{(m)} + \cdots + \lambda_d^{(m)} = r$, then each leg can be partitioned into paths of order $m$ with a connected set of $1+\lambda_1^{(m)} + \cdots + \lambda_d^{(m)} = r$ vertices left over yielding the desired partition. If $1+\lambda_1^{(m)} + \cdots + \lambda_d^{(m)} = m+r$ and $\lambda_i^{(m)} \geq r$ for some $i = 1, \dots, d$, then pick the leg of length $\lambda_i$ such that $\lambda_i^{(m)} \geq r$ and use the $r$ vertices from the end, that is the $r$ vertices forming a path terminating at the leaf, as one set. With this set of $r$ vertices removed, partition the rest of the spider into paths of length $m$ starting at the end of each leg. This leaves $1+\lambda_1^{(m)} + \cdots + (\lambda_i-r)^{(m)} + \cdots + \lambda_d^{(m)} = m + r - r = m$ vertices in the last set and yields the desired connected partition.
    
\end{proof}

Note that Theorem \ref{Th: Dahl-odd} is a specialization of Theorem \ref{Th: mod-test} when $m = 2$. Theorem \ref{Th: mod-test} has many simple yet powerful applications. We state a few particularly interesting ones in the next corollary. Some of the conditions mentioned can be found in \cite{DSvW}.

\begin{coro} \label{Cor: var cond}
Let $S = S(\lambda_1, \ldots, \lambda_d)$ be an $n$-vertex spider. If any of the following conditions hold, then $S$ is missing a connected partition of some type and is not $e$\nobreakdash-positive.
\begin{enumerate}
    \item $\lambda_i < \lambda_{i+1} + \cdots + \lambda_d$ for some $i$. When $i = 1$, this condition is equivalent to $\lambda_1 < \floor{\frac{n}{2}}$. When, $\lambda_i \geq 2$ and $i \geq 2$, \cite[Lemma 28]{DSvW} gives a slightly stronger condition with $\lambda_i \leq \lambda_{i+1} + \cdots + \lambda_d$.
    \item At least $2m-1$ legs have length not divisible by $m$.
    \item At least $m$ legs have length not divisible by $m$, where $m \mid n$
    \item $m | n$, and $\lambda_i + 1 \leq m \leq \lambda_i + \cdots + \lambda_d$ for some $i$.
    \item $\lambda_i +1$ and $\lambda_j +1$ have a common factor $m > 1$, and $m \nmid \lambda_k$, for distinct $i, j,$ and $k$.
    \item $n^{(t_i)} > \lambda_i$ for some $i$, where $t_i = \lambda_i + \cdots + \lambda_d$ and $n^{(t_i)}$ denotes the residue of $n$ modulo $t_i$.
\end{enumerate} 
\end{coro}
\begin{proof}
    We briefly explain the application of Theorem \ref{Th: mod-test} that results in each item.
    
    For item 1, apply Theorem \ref{Th: mod-test} with $m = \lambda_i + 1$. Then $1 + \lambda_1^{(m)} + \cdots + \lambda_d^{(m)} \geq 1 + \lambda_i + \lambda_{i+1} + \cdots + \lambda_{d} \geq 2m$. 
    
    For item 2, apply Theorem \ref{Th: mod-test} with the given $m$ and note that $1 + \lambda_1^{(m)} + \cdots + \lambda_d^{(m)} \geq 2m$.
    
    For item 3, apply Theorem \ref{Th: mod-test} with the given $m$ and note that $1 + \lambda_1^{(m)} + \cdots + \lambda_d^{(m)} \geq m + 1$, so neither condition in Theorem \ref{Th: mod-test} can hold as $n^{(m)} = 0$.
    
    For item 4, apply Theorem \ref{Th: mod-test} with the given $m$. Then, $1 + \lambda_1^{(m)} + \cdots + \lambda_d^{(m)} \geq 1 + \lambda_i + \cdots + \lambda_d \geq m+1$, and neither condition in Theorem \ref{Th: mod-test} can hold as $n^{(m)} = 0$.
    
    For item 5, apply Theorem \ref{Th: mod-test} with the given $m$. Then, $\lambda_i^{(m)} = m-1$, $\lambda_j^{(m)} = m-1$, $\lambda_k^{(m)} \geq 1$, and $1 + \lambda_1^{(m)} + \cdots + \lambda_d^{(m)} \geq 2m$.
    
    For item 6, apply Theorem \ref{Th: mod-test} with $m = t_i = \lambda_i + \cdots + \lambda_d$. First note that $i \neq 1$, as otherwise, $n^{(t_i)} = 1 \leq \lambda_1$. Thus, $i \geq 2$ and
    \begin{equation*}
        1+\lambda_1^{(t_i)}+\cdots + \lambda_d^{(t_i)} = 1 + \lambda_1^{(t_i)} + \cdots + \lambda_{i-1}^{(t_i)} + t_i > n^{(t_i)}.
    \end{equation*}
    The only way for $S$ satisfy the conditions in Theorem~\ref{Th: mod-test} is to have $1 + \lambda_1^{(t_i)} + \cdots + \lambda_{i-1}^{(t_i)} + t_i =  n^{(t_i)} + t_i$ and $\lambda_j^{(t_i)} \geq n^{(t_i)}$ for some $1 \leq j \leq d$. By the assumption that $n^{(t_i)} > \lambda_i$, the latter inequality can only be true for $1 \leq j \leq i-1$, but if this is the case then the former equality cannot be true.
\end{proof}
While Theorem~\ref{Th: mod-test} and Corollary~\ref{Cor: var cond} give many simple conditions that show a spider is not $e$\nobreakdash-positive, we note that there are arbitrarily large spiders and spiders with arbitrarily many legs, for which these conditions are not sufficient. For example, one can check that the spider $S(k,6,2,1)$ satisfies the conditions of Theorem~\ref{Th: mod-test} with every positive integer $m$ as long as $k$ is a positive integer divisible by $10!$. More generally, if $S(\lambda_1, \lambda_2, \ldots, \lambda_d)$ is an $n$-vertex spider that satisfies Theorem~\ref{Th: mod-test} with every $m$, it is not hard to see that $S(k, \lambda_1, \lambda_2, \ldots, \lambda_d)$ does so as well as long as $k$ is divisible by $n!$.

It follows that we need a different condition for $e$\nobreakdash-positivity to obtain our main result.

\begin{theo} \label{Th: qm-test}
Let $S = S(\lambda_1, \ldots, \lambda_d)$ be an $n$-vertex spider. For some $\lambda_i$ such that $2 \leq i < d$, let $a = \ceil{\frac{\lambda_i + 1}{m}}$ where $m$ is a positive integer, $n = qa + r$ where $0 \leq r < a$, $r = qd' + r'$ where $0 \leq r' < q$, and $t = \lambda_{i+1} + \cdots + \lambda_d$. If $q = \floor*{\frac{n}{a}}> \frac{m(a-1)}{t - 2m + 1}$, $t > 2m-1$, and $a > \lambda_{i+1}$, then $S$ is missing a connected partition of some type and is not $e$\nobreakdash-positive.
\end{theo}

\begin{proof}
    We show that $S$ is missing a connected partition of type $(a+ d' + 1)^{r'}(a+ d')^{q-r'}$.

    Suppose otherwise, that $S$ does contain a connected partition, $C$, of this type. Since $a + d' \geq a > \lambda_{i+1}, \ldots, \lambda_d$, one set in $C$ must contain the center along with each of the legs of lengths $\lambda_{i+1}, \ldots, \lambda_d$. Consider the number of sets required to contain all of the vertices from the leg of length $\lambda_i$. Suppose $k$ sets are used. Then only one set can contain the center and vertices from other legs and the remaining $k-1$ sets contain at most $\lambda_i$ vertices collectively, so $(k-1)(a+d') \leq \lambda_i$ vertices. Notice that $m(a+d') \geq ma > \lambda_i$, so $k-1 < m$ and $k \leq m$ with equality only if one of the $k$ sets contains the center.

    Putting the preceding results together, it follows that the legs of lengths $\lambda_i, \lambda_{i+1}, \ldots, \lambda_d$ are contained in the union of at most $m$ sets. Hence for some $V_1, \ldots, V_m \in C$, $|V_1| + \cdots + |V_m| \geq 1 + \lambda_i + \cdots +\lambda_d = 1 + \lambda_i + t$. However, we claim that this value is greater than $m(a+d'+1)$. Indeed,
    \begin{equation*}
        m(a+d'+1) = ma + m + m\frac{r-r'}{q} \leq ma + m + \frac{m(a-1)}{q}.
    \end{equation*}

    By assumption $\frac{m(a-1)}{q} < t-2m+1$, and $a < \frac{\lambda_i+1}{m} + 1$. The second inequality implies $ma \leq \lambda_i + m$ as $ma$ is an integer. It follows that, 
    \begin{equation*}
        m(a+d'+1) < \lambda_i +m + m + t-2m+1 = 1+ \lambda_i + t,
    \end{equation*}
 and the legs of lengths  $\lambda_i, \lambda_{i+1}, \ldots, \lambda_d$ cannot be contained in the union of at most $m$ sets. As a result, $S$ cannot have a connected partition of the stated type. 
\end{proof}

Theorem \ref{Dahl-q}, restated from \cite{DSvW}, is a special case of Theorem \ref{Th: qm-test} when $m = 1$. In Example~\ref{ex: spider th 3.6} we give a spider where Theorem \ref{Th: qm-test} shows it is not $e$\nobreakdash-positive while Theorem \ref{Dahl-q} is not sufficient.
%  In general, when $\lambda_i$ is large, it is optimal to apply Theorem \ref{Th: qm-test} for some $m>1$. The best bound occurs when $m$ is around $t/2$, but it is not guaranteed that the condition  $\ceil{\frac{\lambda_i+1}{m}} > \lambda_{i+1}$ is satisfied in this case. 
\begin{ex} \label{ex: spider th 3.6}
Let $S = S(448, 276, 90, 1, 1)$ be a spider with $817$ vertices. Then it is easy to see that $S$ does not satisfy the conditions of Theorem \ref{Dahl-q}. However, take $m = 3$, $a = \ceil*{\frac{277}{3}} = 93$, and $t = 90 + 1 + 1 = 92$. Then indeed $t > 2m - 1$, $a > 90$, and $q = \floor*{\frac{817}{93}} = 8 > \frac{3(93 - 1)}{92 - 6 + 1} = \frac{92}{29}$, and $S$ is missing a connected partition of type $(103, 102^7)$.
\end{ex}
 
Theorems \ref{Dahl-q} and \ref{Th: qm-test} also imply a sort of geometric progression in the legs of any $e$\nobreakdash-positive spider as demonstrated by the following corollary. 

\begin{coro} \label{Cor: sqrt bd}
If $S = S(\lambda_1, \ldots, \lambda_d)$ is spider with $d \geq 5$ and has a connected partition of every type, then $\lambda_i+1 > \sqrt{\frac{n}{2}(\lambda_{i+1}+1)}$ for $2 \leq i \leq d-2$. If $i > 2$ then this can be improved to $\lambda_i > \sqrt{\frac{n}{2}\lambda_{i+1}}$.
\end{coro}
\begin{proof}
    This is a straightforward application of Theorem \ref{Dahl-q}. Since $S$ is has a connected partition of every type, Theorem \ref{Dahl-q} implies that $q = \floor{\frac{n}{\lambda_i+1}} < \frac{\lambda_i+1}{t-1}$, where $t = \lambda_{i+1} + \cdots + \lambda_{d} \geq \lambda_{i+1} + d - i$. By assumption, $i \leq d-2$, so $t-1 \geq \lambda_{i+1} + 1$. Finally, since $n > \lambda_i + 1$, it follows that $\floor{\frac{n}{\lambda_i+1}} > \frac{n}{2(\lambda_i+1)}$. Altogether this shows,
    \begin{equation*}
        \frac{n}{2(\lambda_i+1)} < \floor*{\frac{n}{\lambda_i+1}} < \frac{\lambda_i+1}{t-1} \leq \frac{\lambda_i+1}{\lambda_{i+1}+1} \implies \lambda_i+1 > \sqrt{\frac{n}{2}(\lambda_{i+1}+1)}.
    \end{equation*}
    
    For the second part, if $i > 2$, the first part of Corollary \ref{Cor: var cond} implies that $\lambda_i \geq \lambda_{d-2} \geq 2$ and $n > \lambda_1 + \lambda_2 + \lambda_i \geq 4\lambda_i$. If, $\lambda_i \geq 3$, then it follows that $\frac{n}{\lambda_i} > 4 \geq 2\frac{\lambda_i+1}{\lambda_i - 1}$. If $\lambda_i = 2$, then applying the first part of Corollary~\ref{Cor: var cond} twice yields $\lambda_2 \geq 4$ and $\lambda_1 \geq 8$, and in particular, $n \geq 17 > 2\frac{\lambda_i+1}{\lambda_i - 1}$. Either way, we may conclude that $n > 2\frac{\lambda_i+1}{\lambda_i - 1}$ and as a result,
    \begin{equation*}
        \frac{n}{2\lambda_i} \leq  \frac{n}{\lambda_i+1} - 1 < \floor*{\frac{n}{\lambda_i+1}} \leq \frac{\lambda_i}{t-1} < \frac{\lambda_i}{\lambda_{i+1}},
    \end{equation*}    
    where the third inequality is due to Theorem \ref{Th: qm-test} with $m = 1$. As a result,
    \begin{equation*}
    \lambda_i > \sqrt{\frac{n}{2}\lambda_{i+1}}.
    \end{equation*}
\end{proof}

Using Corollary \ref{Cor: sqrt bd} we improve the bound of Theorem \ref{dahl bd}, restated from \cite{DSvW}, and show that spiders with at least six legs are not $e$\nobreakdash-positive.

\begin{theo} \label{new bd}
Let $S = S(\lambda_1, \ldots, \lambda_d)$ be an $n$-vertex spider, with $d \geq 5$. If $S$ has a connected partition of every type then $(n/2)^{-1/2}+\cdots+(n/2)^{-1/2^{d-3}} < 1$.
\end{theo}

\begin{proof}
    Suppose $S$ is $e$\nobreakdash-positive. Applying Corollary \ref{Cor: sqrt bd} yields the following inequalities,
    \begin{align*}
        \lambda_{d-2} &> (n/2)^{1/2}, \\
        \lambda_{d-3} &> (n/2)^{1/2}\lambda_{d-2}^{1/2} > (n/2)^{1-1/2^2}, \\
        &\vdots \\
        \lambda_3 &> (n/2)^{1-1/2^{d-4}}, \\
        \lambda_2 + 1 &> (n/2)^{1/2}(\lambda_{3}+1)^{1/2} > (n/2)^{1/2}(\lambda_{3})^{1/2}  > (n/2)^{1-1/2^{d-3}}.
    \end{align*}
    
Adding these together and applying the first part of Corollary \ref{Cor: var cond} yields, 
\begin{equation*}
    1 + \lambda_{d-2}+ \cdots + \lambda_{2} < n/2 \implies (n/2)^{1/2} + \cdots + (n/2)^{1-1/2^{d-3}} < n/2,
\end{equation*}
which can be simplified to
\begin{equation*}
    (n/2)^{-1/2}+\cdots+(n/2)^{-1/2^{d-3}} < 1.
\end{equation*}
\end{proof}

\begin{theo} \label{new result}
If $S = S(\lambda_1 , \ldots, \lambda_d)$ is a spider with $d \geq 6$ legs, then $S$ is missing a connected partition of some type and is not $e$\nobreakdash-positive.
\end{theo}

\begin{proof}
    Set $m = \floor*{\frac{\lambda_2 + 1}{\lambda_3 + 1}}$, $a=\ceil{\frac{\lambda_2+1}{m}}$, and $t = \lambda_3 + \cdots + \lambda_d$.  If $S$ has a connected partition of every type, then recall that by the first part of Corollary \ref{Cor: var cond}, the longest leg of $S$ has at least $\floor*{\frac{n}{2}}$ vertices, and as a result $\lambda_2 + t \leq n/2$ and in particular $n/2 > n/2 - (\lambda_3 + 1) >\lambda_2 + 1$. Also, by Corollary \ref{Cor: sqrt bd},  $\lambda_4 > (n/2)^{1/2}$ and $\lambda_3 > (n/2)^{3/4}$. These inequalities imply that $m \leq \frac{\lambda_2 + 1}{\lambda_3 + 1} < \frac{n/2}{(n/2)^{3/4}} = (n/2)^{1/4}$. It is easy to see that by the bound in Theorem \ref{new bd}, if $S$ has at least six legs and is $e$\nobreakdash-positive, then $n$ must be large enough so that $2(n/2)^{1/4} < (n/2)^{1/2} < \lambda_4$, and as a result, $t - 2m + 1 > \lambda_3 + 1$. 
    
    Finally, note that $m(a-1) = ma - m \leq \lambda_2 + m - m = \lambda_2$, which, with the previous inequality, yields
    \begin{equation*}
        \frac{m(a-1)}{t-2m+1} < \frac{\lambda_2}{\lambda_3 + 1}.
    \end{equation*}
    
    However, we claim that this violates the conditions of Theorem \ref{Th: qm-test} with $\lambda_2, m, a,$ and $t$ as defined above, contradicting the $e$\nobreakdash-positivity of $S$ (or more specifically that $S$ has a connected partition of every type).
    
    Indeed, $\lambda_2 + 1 \geq \lambda_3 + 1$ implies that $m > \frac{\lambda_2 + 1}{2(\lambda_3 + 1)}$, so $a < \frac{\lambda_2 + 1}{m} + 1 < 2(\lambda_3+1) + 1$. As $a$ is an integer, this means that $a \leq 2(\lambda_3 + 1)$. Thus,  
    
    \begin{equation*}
        \floor*{\frac{n}{a}} > \frac{n}{a} - 1 \geq \frac{n}{2(\lambda_3+1)} - 1= \frac{n/2 - (\lambda_3+1)}{\lambda_3+1} > \frac{\lambda_2}{\lambda_3 + 1},
    \end{equation*}
    which implies $\floor*{\frac{n}{a}} > \frac{m(a-1)}{t-2m+1}$. It remains to check that $t > 2m-1$ and $ a > \lambda_{3}$. The first inequality is true as $t - 2m + 1 > \lambda_3 + 1 > 0$, and the second is true as $a \geq \frac{\lambda_2 + 1}{m} \geq \frac{\lambda_2+1}{(\lambda_2+1)/(\lambda_3+1)} > \lambda_3$. Therefore any spider with at least six legs is missing a connected partition of some type and is not $e$\nobreakdash-positive.
\end{proof}

Finally, all of the results in this section can be extended to trees and graphs with cut vertices by applying \cite[Lemmas 12 and 13]{DSvW}. In particular, combining \cite[Lemmas 12 and 13]{DSvW} with Theorem~\ref{new result} yields Theorems~\ref{new result trees} and \ref{new result cut}. We restate \cite[Lemma 12]{DSvW} regarding trees below. The other lemma regarding cut vertices is similar.

\begin{lemma} \label{lm: dahl trees} (Dahlberg, She, and van Willigenburg \cite{DSvW})
Let $T$ be a tree with a vertex of degree $d \geq 3$, and let $v$ be any such vertex. Let $t_1, \ldots, t_d$ denote number of vertices of the subtrees $T_1, \ldots, T_d$ rooted at the $d$ vertices adjacent to $v$ respectively. If $T$ has a connected partition of type $\lambda$, then the spider $S(t_1, \ldots, t_d)$ has a connected partition of type $\lambda$ as well.
\end{lemma}

\begin{theo} \label{new result trees}
Any tree with a vertex of degree at least six is not $e$\nobreakdash-positive.
\end{theo}

\begin{theo}\label{new result cut}
Any graph with a cut vertex whose deletion produces a graph with at least six connected components is not $e$\nobreakdash-positive.
\end{theo}

Theorem \ref{new result trees} resolves a conjecture of Dahlberg, She, and van Willigenburg for $d \geq 6$ \cite{DSvW}. Regarding the tightness of this result, it is mentioned in \cite{DSvW} that the spider $S(6,4,1,1)$ has a connected partition of every type. Therefore, only the remaining $d = 5$ case in Theorem~\ref{new result trees} can be resolved via the method of reducing to spiders and showing that a connected partition of some type is missing. We are unaware of any spiders with five legs that have a connected partition of every type. 

For the $d=4$ case, we were able to find a number of spiders with four legs, in addition to  $S(6,4,1,1)$, that have a connected partition of every type. We discuss these spiders in the next section and show that they are not $e$\nobreakdash-positive. However, our result will proceed by directly calculating $e$\nobreakdash-coefficients, so it does not tell us anything about the trees corresponding to these spiders as we cannot invoke Lemma~\ref{lm: dahl trees}.
\section{Spiders with few legs} \label{sec: small spiders}

The aim of this section, and the next, is to capitalize on the simple structure of spiders and better understand their $e$\nobreakdash-positivity through direct calculations of certain $e$\nobreakdash-coefficients. We focus on spiders with four legs in this section and turn to spiders with more legs in the next. 

While we previously gave many conditions that must be satisfied for a spider to be $e$\nobreakdash-positive, we find that there are still many spiders with fewer than six legs that pass all of the tests. For example, using a brute force search we discovered a number of four legged spiders of the form $S(a, b, 2, 1)$ that satisfy the conditions for $e$\nobreakdash-positivity of the previous section. For the spiders $S(15,12,2,1),$ $S(16,12,2,1),$ and $S(21, 12, 2, 1)$, we were able to further verify that each indeed has a connected partition of every type. Thus, to show that these spiders are not $e$\nobreakdash-positive we must resort to calculating $e$\nobreakdash-coefficients.

In Theorem~\ref{cor: 4 leg q}, we will show that no spider of the form $S(a, b, 2, 1)$, is $e$\nobreakdash-positive. Spiders of the form $(a,b,1,1)$ with connected partitions of every type were noted in \cite{DSvW} and shown to be, in general, not $e$\nobreakdash-positive. 

Our main tool for calculating $e$\nobreakdash-coefficients will be the decomposition shown in Lemma~\ref{lm: spi decomp}, and an expression for the $e$\nobreakdash-coefficients of chromatic symmetric functions of paths \cite{Wolfe}. The decomposition is obtained by repeatedly applying the Triple-Deletion rule \cite{OS} and a trivial property of chromatic symmetric functions of graphs with disjoint subgraphs \cite{S95}. These four lemmas are stated next
% In practice, we find that it is often possible to observe a pattern in the $e$\nobreakdash-coefficients of the chromatic symmetric functions of spiders, conjecture an expression for these coefficients, and then inductively prove the expression holds with Lemma~\ref{lm: spi decomp}.

\begin{lemma} [\!\! Stanley \cite{S95}]\label{lm: disj}
For disjoint graphs $G$ and $H$, $X_{G \cup H} = X_G X_H$.
\end{lemma}

\begin{lemma} [\!\! Triple-Deletion Rule of Orellana and Scott \cite{OS}] \label{lm: tri del}
Let $G(V, E)$ be a graph with vertices $v, v_1, v_2$ where $e_1 = vv_1 \in E$, $e_2 = vv_2 \in E$ and $e_3 = v_1v_2 \notin E$, and let $S = \{e_1, e_2, e_3\}$. Let $G_{A} = G(V, (E - S) \cup A)$ where $A \subseteq S$. Then, 
\begin{equation*}
    X_{G} = X_{G_{\{e_2,e_3\}}} + X_{G_{\{e_1\}}} - X_{G_{\{e_3\}}}.
\end{equation*}
\end{lemma}

Figure~\ref{spi dec fig} illustrates Lemma~\ref{lm: tri del} when applied to a spider.

\begin{lemma}  [\!\! Wolfe \cite{Wolfe}] \label{lm: Wolfe} 
Let $\lambda = (1^{a_1}, \ldots, n^{a_n})$ be a partition of $n$ and let $P_n$ be the path on $n$ vertices. Then the coefficient of $e_\lambda$ in the expansion of $X_{P_n}$ is given by

\begin{align*}
       [e_\lambda]X_{P_n} = &\binom{a_1+ \cdots + a_n}{a_1, \dots , a_n} \prod_{j = 1}^{n} (j-1)^{a_j} + \\
       & \sum_{i = 1}^{n}  \left( \binom{(a_1+ \cdots + a_n) - 1}{a_1, \dots , a_i-1, \dots , a_n} 
       \left( \prod_{j = 1, j \neq i}^{n} (j-1)^{a_j} \right) (i-1)^{a_i-1}
       \right).
\end{align*}
\end{lemma}

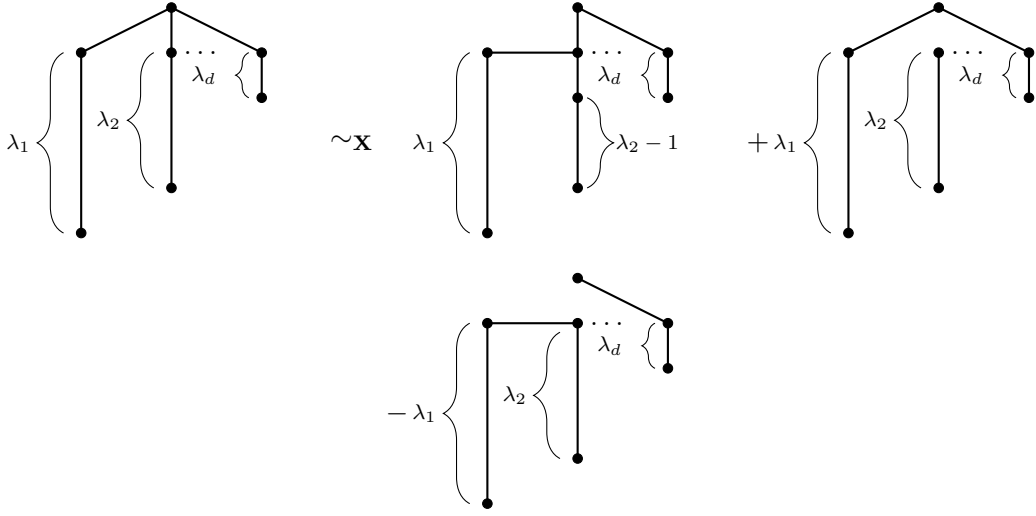
\begin{figure} 
    \centering
    \begin{center}
    \begin{tikzpicture}[scale=0.60] \label{spi dec fig}
    \pgfmathsetmacro {\xa }{-10}
    \pgfmathsetmacro {\ya }{0}
    \pgfmathsetmacro {\xeps }{4}

    %% one
    \draw[fill=black] (\xa, \ya) circle (3pt);
    \draw[fill=black]  (\xa - 2, \ya - 1) circle (3pt);
    \draw[fill=black]  (\xa - 2, \ya -5) circle (3pt);
    
    \draw[fill=black] (\xa,\ya-1) circle (3pt);
    \draw[fill=black] (\xa,\ya-4) circle (3pt);
    
    \draw[fill=black] (\xa+2,\ya-1) circle (3pt);
    \draw[fill=black] (\xa+2,\ya-2) circle (3pt);
    
    \draw[thick] (\xa,\ya) -- (\xa - 2, \ya - 1) -- (\xa - 2, \ya -5) (\xa, \ya) -- (\xa,\ya-1) -- (\xa,\ya-4) (\xa,\ya)--(\xa+2,\ya-1)--(\xa+2,\ya-2);
    
    \node at (\xa + 0.7,\ya -1) {\ldots};
    \draw [decorate,decoration={brace,amplitude=10pt},xshift=-4pt,yshift=0pt] 
    (\xa-2-0.25, \ya - 5) -- (\xa-2-0.25, \ya - 1) node [black,midway,xshift=-0.6cm]{\scriptsize $\lambda_1$};
    \draw [decorate,decoration={brace,amplitude=10pt},xshift=-4pt,yshift=0pt]
    (\xa-0.25, \ya - 4) -- (\xa-0.25,\ya-1) node [black,midway,xshift=-0.6cm]{\scriptsize $\lambda_2$};
    \draw [decorate,decoration={brace,amplitude=5pt},xshift=-1pt,yshift=0pt]
    (\xa +2-0.25, \ya - 2) -- (\xa + 2-0.25,\ya - 1) node [black,midway,xshift=-0.6cm]{\scriptsize $\lambda_d$};
    \node at (\xa + \xeps, -3) {$\sim_\mathbf{X}$};
%%%%%%%%%%%%%%%%%%
    \pgfmathsetmacro {\xa }{\xa + 2* \xeps + 1}

    \draw[fill=black] (\xa, \ya) circle (3pt);
    \draw[fill=black]  (\xa - 2, \ya - 1) circle (3pt);
    \draw[fill=black]  (\xa - 2, \ya -5) circle (3pt);
    
    \draw[fill=black] (\xa,\ya-1) circle (3pt);
    \draw[fill=black] (\xa,\ya-2) circle (3pt);

    \draw[fill=black] (\xa,\ya-4) circle (3pt);
    
    \draw[fill=black] (\xa+2,\ya-1) circle (3pt);
    \draw[fill=black] (\xa+2,\ya-2) circle (3pt);
    
    \draw[thick] (\xa - 2, \ya - 1) -- (\xa,\ya-1) (\xa - 2, \ya - 1) -- (\xa - 2, \ya -5) (\xa, \ya) -- (\xa,\ya-1) -- (\xa,\ya-4) (\xa,\ya)--(\xa+2,\ya-1)--(\xa+2,\ya-2);
    
    \node at (\xa + 0.7,\ya -1) {\ldots};

    \draw [decorate,decoration={brace,amplitude=10pt},xshift=-4pt,yshift=0pt] 
    (\xa-2-0.25, \ya - 5) -- (\xa-2-0.25, \ya - 1) node [black,midway,xshift=-0.6cm]{\scriptsize $\lambda_1$};
    \draw [decorate,decoration={brace,amplitude=10pt},xshift=-4pt,yshift=0pt]
    (\xa+0.35, \ya - 2) -- (\xa+0.35,\ya-4) node [black,midway,xshift=0.8cm]{\scriptsize $\lambda_2-1$};
    \draw [decorate,decoration={brace,amplitude=5pt},xshift=-1pt,yshift=0pt]
    (\xa + 2 -0.25, \ya - 2) -- (\xa + 2 -0.25,\ya - 1) node [black,midway,xshift=-0.6cm]{\scriptsize $\lambda_d$};
    \node at (\xa + \xeps, -3) {$\mathbf{+}$};

%%%%%%%%%%%%%%%%%%%%%%
    \pgfmathsetmacro {\xa }{\xa + 2* \xeps}

    \draw[fill=black] (\xa, \ya) circle (3pt);
    \draw[fill=black]  (\xa - 2, \ya - 1) circle (3pt);
    \draw[fill=black]  (\xa - 2, \ya -5) circle (3pt);
    
    \draw[fill=black] (\xa,\ya-1) circle (3pt);
    \draw[fill=black] (\xa,\ya-4) circle (3pt);
    
    \draw[fill=black] (\xa+2,\ya-1) circle (3pt);
    \draw[fill=black] (\xa+2,\ya-2) circle (3pt);
    
    \draw[thick] (\xa,\ya) -- (\xa - 2, \ya - 1) -- (\xa - 2, \ya -5) (\xa,\ya-1) -- (\xa,\ya-4) (\xa,\ya)--(\xa+2,\ya-1)--(\xa+2,\ya-2);
    
    \node at (\xa + 0.7,\ya -1) {\ldots};
    \draw [decorate,decoration={brace,amplitude=10pt},xshift=-4pt,yshift=0pt] 
    (\xa-2-0.25, \ya - 5) -- (\xa-2-0.25, \ya - 1) node [black,midway,xshift=-0.6cm]{\scriptsize $\lambda_1$};
    \draw [decorate,decoration={brace,amplitude=10pt},xshift=-4pt,yshift=0pt]
    (\xa-0.25, \ya - 4) -- (\xa-0.25,\ya-1) node [black,midway,xshift=-0.6cm]{\scriptsize $\lambda_2$};
    \draw [decorate,decoration={brace,amplitude=5pt},xshift=-1pt,yshift=0pt]
    (\xa + 2 -0.25, \ya - 2) -- (\xa + 2 -0.25,\ya - 1) node [black,midway,xshift=-0.6cm]{\scriptsize $\lambda_d$};
%%%%%%%%%%%%%%%%%%%%%%
    \pgfmathsetmacro {\xa }{\xa-2*\xeps}
    \pgfmathsetmacro {\ya }{\ya - 6}
    
    \node at (\xa - \xeps, -3 - 6) {$\mathbf{-}$};

    \draw[fill=black] (\xa, \ya) circle (3pt);
    \draw[fill=black]  (\xa - 2, \ya - 1) circle (3pt);
    \draw[fill=black]  (\xa - 2, \ya -5) circle (3pt);
    
    \draw[fill=black] (\xa,\ya-1) circle (3pt);
    \draw[fill=black] (\xa,\ya-4) circle (3pt);
    
    \draw[fill=black] (\xa+2,\ya-1) circle (3pt);
    \draw[fill=black] (\xa+2,\ya-2) circle (3pt);
    
    \draw[thick] (\xa,\ya-1) -- (\xa - 2, \ya - 1) -- (\xa - 2, \ya -5) (\xa,\ya-1) -- (\xa,\ya-4) (\xa,\ya)--(\xa+2,\ya-1)--(\xa+2,\ya-2);
    
    \node at (\xa + 0.7,\ya -1) {\ldots};
     
    \draw [decorate,decoration={brace,amplitude=10pt},xshift=-4pt,yshift=0pt] 
    (\xa-2-0.25, \ya - 5) -- (\xa-2-0.25, \ya - 1) node [black,midway,xshift=-0.6cm]{\scriptsize $\lambda_1$};
    \draw [decorate,decoration={brace,amplitude=10pt},xshift=-4pt,yshift=0pt]
    (\xa-0.25, \ya - 4) -- (\xa-0.25,\ya-1.2) node [black,midway,xshift=-0.6cm]{\scriptsize $\lambda_2$};
    \draw [decorate,decoration={brace,amplitude=5pt},xshift=-1pt,yshift=0pt]
    (\xa + 2 -0.25, \ya - 2) -- (\xa + 2 -0.25,\ya - 1) node [black,midway,xshift=-0.6cm]{\scriptsize $\lambda_d$};
    \end{tikzpicture}
    \end{center}
    
    \caption{ An illustration of Lemma \ref{lm: tri del} applied to a spider with $d$ legs. Here, $\sim_{\mathbf{X}}$ denotes that the chromatic symmetric function of the graph on the left equals the stated combination of chromatic symmetric functions of the graphs on the right. }
\end{figure}

\begin{lemma} \label{lm: spi decomp}
Let $S = S(\lambda_1, \ldots, \lambda_d)$ be a spider with $d$ legs, let $\lambda = (\lambda_1, \ldots, \lambda_d)$, and let $P_k$ denote the path on $k$ vertices. Denote by $\lambda_{-j}(i : m)$ the partition that is a modification of $\lambda$ with the $i$th part replaced by $m$ and the $j$th part deleted. Then, 
\begin{equation*}
    X_S = X_{S(\lambda_{-j}(i : \lambda_i + \lambda_j))} + \sum_{k=0}^{\lambda_j - 1}  \left( X_{S(\lambda_{-j}(i :\lambda_i + k))} X_{P_{\lambda_j - k}} - X_{S(\lambda_{-j}(i : k))}X_{P_{\lambda_i + \lambda_j - k}} \right)
\end{equation*}
In particular, when $S = S(\lambda_1, \lambda_2, \lambda_3)$, is a spider with three legs, we get the following decomposition where $(a,b,c)$ is some permutation of $(\lambda_1, \lambda_2, \lambda_3)$
\begin{equation*}
    X_S = X_{P_n}+ X_{P_{a+b+1}}X_{P_{c}} + \cdots + X_{P_{a+b+c}}X_{P_{1}} - X_{P_{a+1}}X_{P_{b+c}} - \cdots - X_{P_{a+c}}X_{P_{b+1}}.
\end{equation*}
\end{lemma}

\begin{proof}
    This is a direct application of Lemmas \ref{lm: disj} and \ref{lm: tri del}. Starting from the center of the spider, $S$, we apply the Triple-Deletion Rule with $e_1$ as the first edge on the leg of length $\lambda_i$, $e_2$ as the first edge on the leg of length $\lambda_j$, and $e_3$ as the edge forming a triangle with $e_1$ and $e_2$. Notice that this decomposes $X_S$ into a linear combination of three terns. Two of the terms, $X_{S_{\{e_1\}}}$ and $X_{S_{\{e_3\}}}$, are products of a chromatic symmetric function of a spider with fewer legs and a chromatic symmetric function of a path, by Lemma \ref{lm: disj}. We keep these terms. For the third term, $X_{S_{\{e_2,e_3\}}}$, notice that $S_{\{e_2,e_3\}}$ is the following modification of the spider $S$. It has the leg of length $\lambda_i$ attached, not to the center, but instead to the vertex on the leg of length $\lambda_j$ that is adjacent to the center. By applying the Triple-Deletion Rule to $X_{S_{\{e_2,e_3\}}}$, we again obtain two terms that we keep, and one term with the leg of length $\lambda_i$ attached to the second vertex (from the center) on the leg of length $\lambda_j$. Repeating this procedure gives the desired decomposition of $X_S$. An example of the first step, with $i = 1$ and $j = 2$ is shown in Figure \ref{spi dec fig}.
\end{proof}

Lemma~\ref{lm: spi decomp} is especially helpful if we are calculating $[e_\mu]X_S$ for some $\mu$ where all parts are large. For example, suppose all parts of $\mu$ are greater than $\lambda_j$ in Lemma~\ref{lm: spi decomp}. Then notice that each term containing $X_{P_{\lambda_j - k}}$ in the summation has no affect on the coefficient $[e_\mu]X_S$ and can be ignored. This is because when $X_{P_{\lambda_j - k}}$ is expanded in terms of elementary symmetric functions, every nonzero term must contain some $e_m$ where $m$ is a positive integer at most $\lambda_j - k$ and strictly smaller than all parts of $\mu$. 

With this strategy, it is not hard to calculate specific $e$\nobreakdash-coefficients of the chromatic symmetric functions of spiders with three or four legs. A similar calculation is done for spiders of the forms $S(r,1,1)$ and $S(r,s,1,1)$ in \cite{DSvW}. 

\begin{lemma}
Let $S = S(\lambda_1, \lambda_2, \lambda_3, \lambda_4)$ be an $n$-vertex spider with a connected partition of every type, $m = \lambda_3 + \lambda_4$, and
$n = mq + r$ with $1 \leq r < m$. Then $[e_{(m+r, m^{q-1})}]X_S = (m - 1)^{q - 3} (m^3 - m^2 q + m^2 r - 2 m^2 - m q r + m q + m + r)$.
\end{lemma}
\begin{proof}
    This is a direct calculation using Lemmas \ref{lm: Wolfe} and \ref{lm: spi decomp}. Write $\lambda_1 = mq_1 + r_1$ and $\lambda_2 = mq_2 + r_2$ for $0 \leq r_1, r_2 < m$. Since $S$ has a connected partition of every type, the first item of Corollary \ref{Cor: var cond} implies $\lambda_1 \geq \lambda_2 \geq \lambda_3 + \lambda_4 = m$ and $q_1, q_2 \geq 1$, while Theorem~\ref{Th: mod-test} implies $1+r_1+r_2 + \lambda_3 + \lambda_4 = 1 + r_1 + r_2 + m = m + r$ and $1 + r_1 + r_2 = r$. 
    
    Using the decomposition in Lemma \ref{lm: spi decomp} and ignoring terms that do not contribute to the coefficient of $e_{(m^{q-1},m+r)}$, particularly those involving a path whose length is not $r$ modulo $m$, yields,
    \begin{align*}
        [e_{(m^{q-1},m+r)}]X_S &= [e_{(m^{q-1},m+r)}](X_{S(\lambda_1, \lambda_2, m)} - X_{P_{1+\lambda_1+\lambda_2}}X_{P_m}) \\ 
        &= [e_{(m^{q-1}, m+r)}]\left( X_{P_{mq + r}} -  X_{P_{m(q_1+1)}}X_{P_{mq_2+r}} -  X_{P_{m(q_2+1)}}X_{P_{mq_1+r}} \right) \\
        & = (m - 1)^{q - 3} (m^3 - m^2 q + m^2 r - m^2 - m q r + m q + mr + r),
    \end{align*}
    as desired.
\end{proof}

\begin{coro} \label{cor: 4 leg q}
Let $S = S(\lambda_1, \lambda_2, \lambda_3, \lambda_4)$ be an $n$-vertex spider, $m = \lambda_3 + \lambda_4$, and $n = mq + r$ where $0 \leq r < m$. If $q > m$, then $S$ is not $e$\nobreakdash-positive.
\end{coro}

\begin{proof}
    This is immediate from the previous Lemma. If $S$ does not have a partition of every type then it is not $e$\nobreakdash-positive by Lemma \ref{Wolf}. If $r = 0$, then notice $\lambda_1^{(m)} + \lambda_2^{(m)} + \lambda_3^{(m)} + \lambda_4^{(m)} + 1 \geq m + 1 > m$, so $S$ is not $e$\nobreakdash-positive  by Theorem \ref{Th: mod-test}.
    
    Lastly if $S$ has a connected partition of every type and $r > 0$, then the previous lemma can be applied. Writing $q = m + c$ for some $c \geq 1$ yields, 
    \begin{align*}
        [e_{(m+r, m^{q-1})}]X_S &= (m - 1)^{q - 3} (-m^2c -mcr + mc + mr + r) \\
        &= (m - 1)^{q - 3} (r-m(mc-c + cr-r)) < 0,
    \end{align*}
    where the last inequality is due to the fact that $m \geq 2$, $m > r$, and $c \geq 1$. Therefore $S$ is not $e$-positive.
\end{proof}
This result also implies that any spider of the form $S(a,b,2,1)$ with $a+b \geq 8$, is not $e$\nobreakdash-positive. For $a+b < 8$ it is straightforward to verify by calculation that each spider $S(a,b,2,1)$ is not $e$\nobreakdash-positive. It is worth noting that for $S = (\lambda_1, \lambda_2, \lambda_3, 1)$, Theorem \ref{Dahl-q} cannot be applied on the leg of length $\lambda_3$, but Corollary \ref{cor: 4 leg q} solves this problem and provides a similar $e$\nobreakdash-positivity condition. This corollary, along with the other results of Section \ref{sec: e-pos}, allow us to confirm that no spider with four legs and at most $400$ vertices is $e$\nobreakdash-positive by directly checking each of the conditions in this paper on each of the spiders. However, the same cannot be said for trees with a vertex of degree $4$, as in some cases we do not show that a connected partition of some type is missing.

While the focus has mostly been on spiders that are not $e$\nobreakdash-positive, we briefly comment on spiders that are $e$\nobreakdash-positive. Most known families of $e$\nobreakdash-positive graphs are composed of paths, cycles, and complete graphs, so it is rare to find families of $e$\nobreakdash-positive trees. One such family however, are the spiders $S(n, n-1, 1)$, which are shown to be $e$\nobreakdash-positive in \cite{DFvW}, for example. In addition to this family, we can calculate that the spiders $S(6,2,1), S(5,3,2), S(6,4,2), S(8, 6, 2), S(9,7,2), S(9,6,1), S(11,6,1),$ and $S(15, 6, 1)$ are $e$\nobreakdash-positive. We state two conjectured families of $e$\nobreakdash-positive spiders in Section~\ref{sec: fin} that include some of these spiders. We were informed of these families by Aliniaeifard, van Willigenburg, and Wang \cite{AvWW}.

\section{More on the $e$\nobreakdash-positivity of spiders} \label{sec: coeff calcs}

Motivated by calculations on spiders with four legs, we attempt to extend these calculations to spiders with any number of legs. The main result of this section is Theorem \ref{th: odd legs} which gives further conditions on when spiders with two odd length legs are $e$\nobreakdash-positive.

\begin{lemma} \label{lm: mq coef}
Let $S = S(\lambda_1, \dots, \lambda_d)$ be an $n$ vertex spider and suppose $n = mq$ for positive integers $m$ and $q$ with $m > 1$. If $S$ has a connected partition of type $(m^q)$, then, $[e_{(m^q)}]X_S = m(m-1)^{q-1}$.
\end{lemma}

\begin{proof}
    Fix a positive integer $m$. The proof is by induction on $d$, the number of legs of the spider $S$, over all $S$ that have a connected partition of type $(m^q)$ for some positive integer $q$. When $d = 1$ or $2$, note that $S$ is a path and the result is a direct application of Lemma \ref{lm: Wolfe}. Now suppose that $d \geq 3$, and that the result holds for all smaller $d$. Let $S$ be a spider with $d$ legs and a connected partition of type $(m^q)$. Choose two arbitrary legs of lengths $\lambda_i = mk_1 + r_1$ and $\lambda_j = mk_2 + r_2$, where $0 < r_1 < m$ and $0 \leq r_2 < m$. Since $S$ has a connected partition of type $(m^q)$, Theorem \ref{Th: mod-test} implies that $r_1 + r_2 < m$. Expanding $X_S$ with Lemma \ref{lm: spi decomp} yields
    \begin{equation*}
        X_S = X_{S(\lambda_{-j}(i : \lambda_i + \lambda_j))} + \sum_{k=0}^{\lambda_j - 1}  \left( X_{S(\lambda_{-j}(i : \lambda_i + k))} X_{P_{\lambda_j - k}} - X_{S(\lambda_{-j}(i : k))}X_{P_{\lambda_i + \lambda_j - k}} \right).
    \end{equation*}
    
    Notice that the chromatic symmetric functions appearing above are all of spiders with fewer than $d$ legs. As we are only interested in the coefficient of $e_{(m^q)}$, the only relevant terms in the summation above are the $X_{S(\lambda_{-j}(i : \lambda_i + k))} X_{P_{\lambda_j - k}}$ with $\lambda_j - k \equiv 0 \pmod{m}$ and $X_{S(\lambda_{-j}(i : k))}X_{P_{\lambda_i + \lambda_j - k}}$ with $\lambda_i + \lambda_j - k \equiv 0 \pmod{m}$. Notice that in the respective cases, $S(\lambda_{-j}(i : \lambda_i + k))$ and $S(\lambda_{-j}(i : k))$ are spiders with fewer than $d$ legs, have $mq'$ vertices for some positive integers $q'$, and satisfy the conditions of Theorem \ref{Th: mod-test}. In other words, whenever $P_{\lambda_j - k}$ (resp. $P_{\lambda_i + \lambda_j - k}$) is a path of order divisible by $m$, $S(\lambda_{-j}(i : \lambda_i + k))$ (resp. $S(\lambda_{-j}(i : k))$) is a spider with fewer than $d$ legs that has a connected partition with all parts of size $m$. 
    
    The number of terms of the first kind is equal to the number of multiples of $m$ in $\{1, \ldots, mk_2 + r_2\}$ and the number of terms of the second kind is equal to the number of multiples of $m$ in $\{mk_1 + r_1 +1, \ldots , mk_1 + mk_2 + r_1 + r_2 \}$. Since $0 < r_1 \leq r_1 + r_2 < m$, in each case this number is $k_2$.
    
    Finally, for both types of terms, by the inductive assumption, the coefficient of $e_{(m^q)}$ is $m^2(m-1)^{q-2}$, and as there are $k_2$ of both terms, the $e_{(m^q)}$ terms over the summation all cancel out leaving only the $e_{(m^q)}$ term in $X_{S(\lambda_{-j}(i : \lambda_i + \lambda_j))}$. It is easy to check that $S(\lambda_{-j}(i : \lambda_i + \lambda_j))$ has fewer than $d$ legs and has a connected partition of type $(m^q)$. By the inductive assumption,  $[e_{(m^q)}]X_{S(\lambda_{-j}(i : \lambda_i + \lambda_j))} = m(m-1)^{q-1}$, and as a result $[e_{(m^q)}]X_S = m(m-1)^{q-1}$ as well.
\end{proof}

\begin{lemma} \label{lm: 2... coef}
Let $S = S(\lambda_1, \dots, \lambda_d)$ be a spider with an even number of vertices and $j$ legs of odd length. Then, $[e_{(2^{n/2})}]X_S = (-1)^{\frac{j-1}{2}}2$. 
\end{lemma}

\begin{proof}
    First notice that $j$ must be odd in order for $S$ to have an even number of vertices, so $(-1)^{\frac{j-1}{2}}2$ is well defined.
    The proof is by induction on $j$, the number of odd length legs. When $j = 1$, the result is simply a special case of Lemma \ref{lm: mq coef} with $m = 2$. Now suppose that $S$ is an $n$-vertex spider with $j \geq 3$ odd length legs and that the result holds for all spiders with $j-2$ odd length legs. Let $\lambda_i$ and $\lambda_j$ be two of the odd lengths. By Lemma \ref{lm: spi decomp},
    \begin{align*}
        X_S = X_{S(\lambda_{-j}(i : \lambda_i + \lambda_j))} + \sum_{k=0}^{\lambda_j - 1}  \left( X_{S(\lambda_{-j}(i : \lambda_i + k))} X_{P_{\lambda_j - k}} \right) - \sum_{k=0}^{\lambda_j - 1}  \left( X_{S(\lambda_{-j}(i : k))}X_{P_{\lambda_i + \lambda_j - k}} \right).
    \end{align*}
    Since we are only interested in the coefficient of $e_{(2^{n/2})}$, all terms in the first summation where $k$ is even and all terms where $k$ is odd in the second summation can be ignored. The only relevant terms are of the form $X_{S(\lambda_{-j}(i : 2q))}X_{P_{2p}}$ with $2q + 2p = \lambda_i + \lambda_j$. Since $S(\lambda_{-j}(i : 2q))$ is a spider with exactly $j-2$ legs of odd length, by the inductive assumption, $[e_{(2^{n/2})}]X_{S(\lambda_{-j}(i : 2q))}X_{P_{2p}} = (-1)^{\frac{j-3}{2}}4$. Thus, the coefficient of $e_{(2^{n/2})}$ in the expansion of $X_S$ is
    \begin{equation*}
        [e_{(2^{n/2})}]X_S = (-1)^{\frac{j-3}{2}}2 + \left(\frac{\lambda_j - 1}{2}\right) (-1)^{\frac{j-3}{2}}4 -\left(\frac{\lambda_j + 1}{2}\right) (-1)^{\frac{j-3}{2}}4  =  (-1)^{\frac{j-1}{2}}2,
    \end{equation*}
    as desired.
\end{proof}

\begin{theo} \label{th: odd legs}
Let $S = S(\lambda_1, \dots, \lambda_d)$ be an $n$ vertex spider with exactly two legs of odd length. Suppose that $2k_1 + 1$ and $2k_2 + 1$ are the lengths of the odd length legs, while $2k_3, \dots, 2k_d$ are the lengths of the even length legs (if they exist), then
\begin{equation*}
    [e_{(3, 2^{k_1 + \cdots + k_d})}]X_S = 4(k_1 + k_2 - k_3 - \cdots - k_d) + 2d - 1.
\end{equation*}
In particular, if a spider $S(\lambda_1, \dots, \lambda_d)$ has exactly two legs of odd length, $d \geq 4$, and $\lambda_1$ even, then it is not $e$\nobreakdash-positive.
\end{theo}

\begin{proof}
    The proof is by induction as follows. Let $S$ be as described above. When $d = 2$, note that $S$ is a path of order $2k_1 + 2k_2 + 3$, so by Lemma \ref{lm: Wolfe}, $[e_{(3, 2^{k_1+k_2})}]X_S = 2(k_1+k_2+1) + 2(k_1 + k_2) + 1 = 4(k_1 + k_2) + 3$, as desired.
    
    When $d = 3$, $S$ has legs of lengths $2k_1 + 1, 2k_2 + 1,$ and $2k_3$, so an application of the $3$-legged case of Lemma \ref{lm: spi decomp} with $(a,b,c) = (2k_3, 2k_1 + 1, 2k_2 + 1)$ yields
    \begin{align*}
        X_S &= X_{P_{2k_1 + 2k_2 + 2k_3 + 3}} + \sum_{i=1}^{2k_2 + 1} \left( X_{P_{2k_{3} + 2k_{1} + 1 + i}}X_{P_{2k_2+2 - i}} - X_{P_{2k_3 + i}}X_{P_{2k_1 + 2k_2 + 3 - i}} \right) \\
        &=  X_{P_{2k_1 + 2k_2 + 2k_3 + 3}} + \sum_{i=1}^{2k_2 + 1} \left( X_{P_{2k_{3} + 2k_{1} + 1 + i}}X_{P_{2k_2+2 - i}} - X_{P_{2k_2 + 2k_3 + 2 - i}}X_{P_{2k_1 + 1 + i}} \right).
    \end{align*}
    
    The individual terms in the sum are all of the form $X_{P_{2p}}X_{P_{2q+1}}$, where $2p + 2q + 1 = n = 2k_1 + 2k_2 + 2k_3 + 3$. By Lemma \ref{lm: mq coef} and the result on $d = 2$, the coefficient of $[e_{(3, 2^{k_1+k_2+k_3})}]$ is
    \begin{equation*}
        [e_{(3, 2^{k_1+k_2+k_3})}]X_{P_{2p}}X_{P_{2q+1}} = [e_{(2^p)}]X_{P_{2p}} [e_{(3, 2^{q-1})}]X_{P_{2q+1}} = 2(4(q-1) + 3),
    \end{equation*}
    if $q > 0$ and $0$ otherwise, since if $q = 0$ every term of $X_{P_{2p}}X_1$'s expansion contains $e_1$.

    Substituting this into the previous summation yields, 
    \begin{align*}
         [e_{(3, 2^{k_1+k_2+k_3})}]X_S &= 4(k_1 + k_2 + k_3) + 3 + \sum_{\substack{i=1 \\ i \text{ odd}} \newline}^{2k_2} (-8k_3) + \sum_{\substack{i=1 \\ i \text{ even}}}^{2k_2} 8k_3 - 2(4(k_3 - 1) + 3) \\
         &=  4(k_1 + k_2 - k_3) + 5,
    \end{align*}
    as desired. 

    Now, suppose that the result holds for all spiders with fewer than $d$ legs and that $d \geq 4$. For brevity, let $k = k_1 + \cdots + k_d$. By an application of Lemma \ref{lm: spi decomp} with $i$ and $j$ such that $\lambda_i$ and $\lambda_j$ are two legs of even lengths, say $2k_3$ and $2k_4$ respectively, 
    \begin{equation*}
        X_S = X_{S(\lambda_{-j}(i : 2k_3 + 2k_4))} + \sum_{a = 0}^{2k_4-1} \left( X_{S(\lambda_{-j}(i : 2k_3 + a))} X_{P_{2k_4 - a}} - X_{S(\lambda_{-j}(i : a))}X_{P_{2k_3 + 2k_4 - a}} \right).
    \end{equation*}
    
    As with the calculation for $d = 3$, notice that the terms appearing in the sum are of the form $X_{S(\lambda_{-j}(i : 2p))}X_{P_{2q}}$ or, $X_{S(\lambda_{-j}(i : 2p+1))}X_{P_{2q+1}}$. 
    
    When $p,q > 0$ note that $S(\lambda_{-j}(i : 2p))$ is a spider with $d-1$ legs, precisely two of which have odd length and $S(\lambda_{-j}(i : 2p + 1))$ is a spider with exactly $3$ legs of odd length, while $P_{2q}$ and $P_{2q+1}$ are paths of length at least $2$, so by the inductive assumption and Lemmas \ref{lm: mq coef} and \ref{lm: 2... coef} we have
    \begin{alignat*}{3}
         &&[e_{(3, 2^{k})}] X_{S(\lambda_{-j}(i : 2p))}X_{P_{2q}} &=  [e_{(3, 2^{k - q})}] X_{S(\lambda_{-j}(i : 2p))}  [e_{(2^q)}]X_{P_{2q}} \\
         &&\quad &= 2(4(k_1+k_2 - p - \sum_{t=5}^{d} k_t) + 2(d-1)-1), \\
         &&[e_{(3, 2^{k})}] X_{S(\lambda_{-j}(i : 2p+1))}X_{P_{2q+1}} &=  [e_{(2^{k-q+1})}]X_{S(\lambda_{-j}(i : 2p+1))}  [e_{(3, 2^{q-1})}] X_{P_{2q+1}} \\
         && &= -2(4(q-1)+3),
    \end{alignat*}
    when $p, q > 0$.
    
    Finally, the inductive assumption also gives the following coefficients,
    \begin{align*}
        [e_{(3, 2^k)}] X_{S(\lambda_{-j}(i : 2k_3))}X_{P_{2k_4}} &=  2(4(k_1+k_2 - k_3 - \sum_{t=5}^{d} k_t) + 2(d-1) - 1) \\
        [e_{(3, 2^k)}] X_{S(\lambda_{-j}(i : 0))}X_{P_{2k_3+2k_4}} &=  2(4(k_1+k_2 - \sum_{t=5}^{d} k_t) + 2(d-2) - 1) \\
         [e_{(3, 2^k)}] X_{S(\lambda_{-j}(i : 2k_4-1))}X_{P_{2k_3+1}} &= -2(4k_3-1).
    \end{align*}
    
    When we substitute everything into the previous equation, the terms in the summation cancel out nicely; in particular, for $1 \leq a \leq 2k_4-2$, the relevant coefficient of each summand is simply $(-1)^{a+1}8k_3$. Thus we have,
    \begin{align*}
       [e_{(3, 2^{k})}] X_S = \ &4(k_1+k_2 -  \sum_{t=3}^{d} k_t) + 2(d-1) - 1 + 
       \sum_{\substack{a=1 \\ a \text{ odd}}}^{2k_4-2}8k_3 + \sum_{\substack{a=1 \\ a \text{ even}}}^{2k_4-2} (-8k_3) \\ 
       &+2(4(k_1+k_2 - k_3 - \sum_{t=5}^{d} k_t) + 2(d-1) - 1) \\
       &-2(4(k_1+k_2 - \sum_{t=5}^{d} k_t) + 2(d-2) - 1) + 2(4k_3-1) \\
       = \ & 4(k_1+k_2 -k_3 - \cdots - k_d) + 2d-1,
    \end{align*}
    finishing the proof.
\end{proof}

\section{Further Avenues} \label{sec: fin}
In this paper we further investigated $e$\nobreakdash-positivity conditions for spiders and trees and showed that no tree with a vertex of degree at least six is $e$\nobreakdash-positive. While the original motivation for studying spiders was the reduction from trees to spiders noted in \cite{DSvW}, it appears that the $e$\nobreakdash-positivity of spiders is an interesting topic in its own right. Our results in Sections \ref{sec: small spiders} and \ref{sec: coeff calcs} show that the simpler structure of spiders makes it easier to understand certain coefficients in the elementary symmetric function expansion of the chromatic symmetric function. We mention some possible avenues for extending the work in this paper and discuss their connection to previous literature.
 
First, we restate a few questions that were mentioned earlier, at the ends of Sections \ref{sec: e-pos} and \ref{sec: small spiders}.

\begin{quest}
 Are there any $e$\nobreakdash-positive trees that reduce, by Lemma \ref{lm: dahl trees}, to $S(6,4,1,1)$? What about to $S(15,12,2,1), S(16, 12, 2, 1), S(21, 12, 2, 1),$ or $S(42, 36, 4, 1)$? These spiders seem to pass all of the tests mentioned in Section \ref{sec: e-pos}.
 \end{quest}
 
At the end of Section \ref{sec: small spiders} we briefly commented on $e$\nobreakdash-positive spiders with three legs. While there seem to be many $e$\nobreakdash-positive spiders outside of the family $S(n, n-1, 1)$, we are not aware of any other infinite families. We were apprised of the following two related conjectures of Aliniaeifard, van Willigenburg, and Wang via personal communication.

 \begin{conj} [\!\! Aliniaeifard, Wang, and van Willigenburg \cite{AvWW}] \label{ConjvW1}
  The family of spiders $S(2(2m+1), 2m, 1)$ is $e$\nobreakdash-positive.
 \end{conj}

 \begin{conj} [\!\!  Aliniaeifard, Wang, and van Willigenburg \cite{AvWW}] \label{ConjvW2}
  The family of spiders $S(n(n!m+1), n!m, 1)$ is $e$\nobreakdash-positive.
 \end{conj}
 
As evidence supporting the $e$\nobreakdash-positivity of these two families, Aliniaeifard, van Willigenburg, and Wang have confirmed Conjecture \ref{ConjvW1} for $m \leq 11$ and Conjecture \ref{ConjvW2} for $n\leq3$ and $m \leq 2$.

Finally, we have the following question on $e$\nobreakdash-positive trees that are not spiders. Examples of such trees seem even more rare. 
\begin{quest}
 Does there exist an infinite family of $e$\nobreakdash-positive trees that can be reduced to a spider that is not $e$\nobreakdash-positive through Lemma \ref{lm: dahl trees}? 
\end{quest}
 
Regarding the last question, recall that while Dahlberg, She, and van Willigenburg's original conjecture is for all trees, the results in Sections \ref{sec: small spiders} and \ref{sec: coeff calcs} are only applicable to spiders as they deal with direct calculations of coefficients. For example, we find a (finite) number of trees that are $e$\nobreakdash-positive but can be reduced to a spider that is not $e$\nobreakdash-positive using Lemma \ref{lm: dahl trees}.

\begin{ex}
Let $M_n$ denote the tree that consists of a path of length $2n+1$ with two additional vertices adjacent to the $n$th and $n+1$th vertices on the path. Then notice the $n$th vertex has degree $3$ and that there are subtrees of size $n+2, n-1,$ and $1$ rooted at its three neighbors. Then for $n = 1, 2, 4, 5, 7,$ and $8$ $M_n$ is $e$\nobreakdash-positive. However, for $n = 2, 4, 5,$ and $8$ the spider $S(n+2, n-1, 1)$ is not $e$\nobreakdash-positive. For $n = 10$ and $11$, $M_n$ is also not $e$\nobreakdash-positive.
\end{ex}

Returning to the $e$\nobreakdash-positivity tests of Section \ref{sec: e-pos}, observe that for all of these tests, if the spider $S = S(\lambda_1, \ldots, \lambda_d)$ passes the test, then a spider, $S'$, obtained by adding two legs together, e.g. replacing two legs of lengths $\lambda_i$ and $\lambda_j$ with a single leg of length $\lambda_i+\lambda_j$, does so as well. This due to the following fact, which is not hard to see.

\begin{prop}
Let $S$ and $S'$ be spiders as defined above. If $S$ has a connected partition of type $\mu$, then $S'$ does as well.
\end{prop}

Thus, for any $e$\nobreakdash-positivity test that relies on finding a missing type of connected partition, if $S$ passes the test, then $S'$ does as well. In fact, the same can be said for the condition in Theorem \ref{th: odd legs}, even though this Theorem does not rely on finding a missing type of connected partition; if $S$ satisfies Theorem \ref{th: odd legs}, then $S'$ does as well. This leads to the following conjecture.

\begin{conj} \label{conj new}
Let $S = S(\lambda_1, \ldots, \lambda_d)$, be a spider with $d$ legs. Let $S'$ be the spider described above. That is, $S'$ has $d-1$ legs and is obtained by combining two of the legs as described above. If $S$ is $e$\nobreakdash-positive then $S'$ is as well. 
\end{conj}

Note that Conjecture \ref{DSvW Conj}, restated from \cite{DSvW}, implies this conjecture. Indeed, if the only $e$\nobreakdash-positive spiders are those with at most three legs, then for every $e$\nobreakdash-positive spider $S$, any $S'$ is a path, which is known to be $e$\nobreakdash-positive (see \cite{S95} for example). However, if there happen to exist spiders with four or five legs that are $e$\nobreakdash-positive, this conjecture would be a natural first step towards understanding which spiders with $4$ or more legs are $e$\nobreakdash-positive. Even if there are no $e$\nobreakdash-positive spiders with four or more legs, we believe that it would be interesting to study any connections between the chromatic symmetric functions of $S$ and $S'$, and that the inductive approach in Section \ref{sec: coeff calcs} could be helpful.

This is also similar to the notion of a graph being \emph{strongly e-positive} presented in \cite{Foley}, where a graph is said to be \emph{strongly $e$\nobreakdash-positive} if all of its induced subgraphs are $e$\nobreakdash-positive. For spiders $S(\lambda)$ and $S(\mu)$, we are replacing the notion of $S(\mu)$ being an \emph{induced subgraph} of $S(\lambda)$ with $\lambda$ being a \emph{refinement} of $\mu$, i.e. a partition formed by adding up parts of $\mu$. 

Finally, we note that partitions of spiders are related to partitions of their line graphs, which makes sense as partitioning the vertices of spiders is very similar to partitioning their edges. In fact the following is true.

\begin{prop}
If a spider $S = S(\lambda_1, \ldots, \lambda_d)$ has a connected partition of every type, then the line graph of $S$, denoted by $S_L$, does as well.
\end{prop}

It is then natural to consider the following conjecture, which we can check is true for all of the $e$\nobreakdash-positive spiders that we know of --- namely those mentioned at the end of Section \ref{sec: small spiders}.
\begin{conj} \label{conj: line graphs}
If a spider $S$ is $e$\nobreakdash-positive, then its line graph, $S_L$, is as well.
\end{conj}

This question has connections to both Stanley's original observation that $e$\nobreakdash-positivity is related to being claw free, and a class of graphs shown to be $e$\nobreakdash-positive by \cite{GS}.

Regarding the first connection, Stanley noted that being claw free was related to being $e$\nobreakdash-positive, but that the conditions were not equivalent. Stanley's smallest example of a non-$e$\nobreakdash-positive graph with no induced claw was the net, or equivalently $S_L$ where $S$ is the non-$e$\nobreakdash-positive spider $S = S(2,2,2)$. Thus, if being claw free is related to being $e$\nobreakdash-positive, then one might expect that taking the line graph of a spider, an operation that removes an induced claw, preserves $e$\nobreakdash-positivity.

For the second connection, note that for spiders of the form $S(\lambda_1, \lambda_2,  1, \ldots, 1)$, and in particular the $e$\nobreakdash-positive family $S(n, n-1, 1)$, their line graphs are part of a family of graphs known as $K$-chains, which are shown to be $e$\nobreakdash-positive in Corollary 7.7 of \cite{GS}. A $K$-chain is a graph consisting of a sequence of complete graphs sequentially identified at a single vertex. The line graph of a spider of the form $S(\lambda_1, \lambda_2,  1, \ldots, 1)$ is a chain of multiple $K_2$'s and a single $K_d$. Using the $e$\nobreakdash-positivity of $K$-chains, one can also see that the converse of our conjecture is false. For example, the family of spiders $S(2a+1, 2b+1, 1)$ is not $e$\nobreakdash-positive by Theorem \ref{Th: mod-test}, but have $e$\nobreakdash-positive line graphs as just noted. 

Regarding the line graphs of spiders, it is shown in \cite{FKK} that they are distinguished by their chromatic symmetric functions. As far as we know, the $e$\nobreakdash-positivity of line graphs of spiders has not been directly studied before. Regardless of whether there exist $e$\nobreakdash-positive spiders with $4$ or more legs, it could still be worthwhile to consider Conjecture \ref{conj: line graphs} for spiders with three legs and better understand the connection between the chromatic symmetric functions of spiders and their line graphs. 

In general, for a spider $S$ with $d$ legs, $S_L$ consists of a complete graph of size $d$, with a path coming out of each vertex. Such graphs are sometimes called generalized spiders and any of the $e$\nobreakdash-positivity questions regarding spiders could be studied for these graphs as well.

\section*{Acknowledgements}
This work was done at the University of Minnesota Duluth REU, funded by NSF-DMS Grant 1949884 and NSA Grant H98230-20-1-0009. I thank Joe Gallian for suggesting this problem, organizing the program, and providing comments. I am also grateful to Amanda Burcroff, Colin Defant, and Yelena Mandelshtam for advising the program, to Trajan Hammonds and Caleb Ji for their comments, and to Stephanie van Willigenburg for sharing Conjectures \ref{ConjvW1} and \ref{ConjvW2}. Lastly I thank two anonymous reviewers for their helpful suggestions.

\end{document}